\documentclass{amsart}

\newtheorem{theorem}{Theorem}[section]
\newtheorem{lemma}[theorem]{Lemma}

\theoremstyle{definition}

\newtheorem{example}[theorem]{Example}

\theoremstyle{remark}
\newtheorem{remark}[theorem]{Remark}

\usepackage{amsmath}
\usepackage{graphicx}
\usepackage{mathrsfs}
\usepackage{algorithm}
\usepackage{algorithmicx}
\usepackage{appendix}
\usepackage{booktabs}
\usepackage{multirow}
\usepackage[colorlinks, linkcolor=blue, anchorcolor=blue, citecolor=blue]{hyperref}
\AtBeginDocument{
                 \label{CorrectFirstPageLabel}
                 
                }

\DeclareMathOperator*{\diag}{diag}
\DeclareMathOperator*{\re}{Re}
\DeclareMathOperator*{\tr}{tr}
\DeclareMathOperator*{\Retr}{Re\,tr}

\numberwithin{equation}{section}

\begin{document}

\title[Spectral variation of a general matrix]{New upper bounds for the spectral variation of a general matrix}

\author{Xuefeng Xu}
\address{Department of Mathematics, Purdue University, West Lafayette, IN 47907, USA}
\email{xuxuefeng@lsec.cc.ac.cn; xu1412@purdue.edu}






\begin{abstract}
Let $A\in\mathbb{C}^{n\times n}$ be a normal matrix with spectrum $\{\lambda_{i}\}_{i=1}^{n}$, and let $\widetilde{A}=A+E\in\mathbb{C}^{n\times n}$ be a perturbed matrix with spectrum $\{\widetilde{\lambda}_{i}\}_{i=1}^{n}$. If $\widetilde{A}$ is still normal, the celebrated Hoffman--Wielandt theorem states that there exists a permutation $\pi$ of $\{1,\ldots,n\}$ such that $\big(\sum_{i=1}^{n}|\widetilde{\lambda}_{\pi(i)}-\lambda_{i}|^{2}\big)^{1/2}\leq\|E\|_{F}$, where $\|\cdot\|_{F}$ denotes the Frobenius norm of a matrix. This theorem reveals the strong stability of the spectrum of a normal matrix. However, if $A$ or $\widetilde{A}$ is non-normal, the Hoffman--Wielandt theorem does not hold in general. In this paper, we present new upper bounds for $\big(\sum_{i=1}^{n}|\widetilde{\lambda}_{\pi(i)}-\lambda_{i}|^{2}\big)^{1/2}$, provided that both $A$ and $\widetilde{A}$ are general matrices. Some of our estimates improve or generalize the existing ones.
\end{abstract}

\maketitle



\section{Introduction}

Let $\mathbb{C}^{m\times n}$ be the set of all $m\times n$ complex matrices, and let $I_{n}$ be the $n\times n$ identity matrix. For any $X\in\mathbb{C}^{m\times n}$, the symbols $X^{\ast}$, $\|X\|_{2}$, and $\|X\|_{F}$ denote the conjugate transpose, the spectral norm, and the Frobenius norm of $X$, respectively. For any $Y\in\mathbb{C}^{n\times n}$, $\tr(Y)$, $\mathcal{D}(Y)$, $\mathcal{L}(Y)$, and $\mathcal{U}(Y)$ stand for its trace, diagonal part, strictly lower triangular part, and strictly upper triangular part, respectively. Furthermore, we define
\begin{equation}\label{delta}
\delta(Y):=\bigg(\|Y\|_{F}^{2}-\frac{1}{n}|\tr(Y)|^{2}\bigg)^{\frac{1}{2}}.
\end{equation}
Obviously, $\delta(Y)\leq\|Y\|_{F}$, and $\delta(Y)=\|Y\|_{F}$ if and only if $\tr(Y)=0$.

Let $A\in\mathbb{C}^{n\times n}$ and $\widetilde{A}=A+E\in\mathbb{C}^{n\times n}$ have the spectra $\{\lambda_{i}\}_{i=1}^{n}$ and $\{\widetilde{\lambda}_{i}\}_{i=1}^{n}$, respectively. For any permutation $\pi$ of $\{1,\ldots,n\}$, we define
\begin{equation}\label{D2}
\mathbb{D}_{2}:=\Bigg(\sum_{i=1}^{n}\big|\widetilde{\lambda}_{\pi(i)}-\lambda_{i}\big|^{2}\Bigg)^{\frac{1}{2}}.
\end{equation}
If $A$ and $\widetilde{A}$ are normal matrices, Hoffman and Wielandt~\cite{Hoffman1953} proved that there exists a permutation $\pi$ of $\{1,\ldots,n\}$ such that
\begin{equation}\label{HW}
\mathbb{D}_{2}\leq\|E\|_{F}.
\end{equation}
This is the well-known \textit{Hoffman--Wielandt theorem}, which reveals the strong stability of the spectrum of a normal matrix. However, the inequality~\eqref{HW} may fail when $A$ or $\widetilde{A}$ is non-normal. Over the past decades, various extensions or analogues of the Hoffman--Wielandt theorem have been developed by many researchers; see, e.g.,~\cite{Elsner1995,Sun1996,Bhatia1997,Eisenstat1998,Ipsen1998,Li1998,Song2002,Li2005,Li2006,Bhatia2007,Li2017,Xu2017}.

If $A\in\mathbb{C}^{n\times n}$ is normal and $\widetilde{A}=A+E\in\mathbb{C}^{n\times n}$ is non-normal, Sun~\cite[Theorem~1.1]{Sun1996} showed that
\begin{equation}\label{Sun}
\mathbb{D}_{2}\leq\sqrt{n}\|E\|_{F}.
\end{equation}
Recently, Xu and Zhang~\cite[Theorem~3.6]{Xu2017} derived that
\begin{equation}\label{XXF}
\mathbb{D}_{2}\leq\sqrt{\|E\|_{F}^{2}+(n-1)\delta(E)^{2}},
\end{equation}
which improved the estimate~\eqref{Sun} due to $\delta(E)\leq\|E\|_{F}$. Nevertheless, the estimates~\eqref{HW}--\eqref{XXF} may be invalid for a general matrix $A$. As is well known, for any $A\in\mathbb{C}^{n\times n}$, there is a nonsingular matrix $Q\in\mathbb{C}^{n\times n}$ such that
\begin{displaymath}
Q^{-1}AQ=\diag\big(J_{1},\ldots,J_{p}\big),
\end{displaymath}
where each $J_{i}\in\mathbb{C}^{m_{i}\times m_{i}}$ ($\sum_{i=1}^{p}m_{i}=n$) is a Jordan block. Let
\begin{displaymath}
m=\max_{1\leq i\leq p}m_{i} \quad \text{and} \quad E_{Q}=Q^{-1}EQ.
\end{displaymath}
It was proved by Song~\cite[Theorem~2.1]{Song2002} that
\begin{equation}\label{Song}
\mathbb{D}_{2}\leq
\begin{cases}
\sqrt{n}\big(\sqrt{n-p}+1\big)\|E_{Q}\|_{F}^{\frac{1}{m}}, &\text{if $\|E_{Q}\|_{F}<1$},\\
\sqrt{n}\big(\sqrt{n-p}+1\big)\|E_{Q}\|_{F}, &\text{if $\|E_{Q}\|_{F}\geq 1$}.
\end{cases}
\end{equation}

In this paper, we establish some new upper bounds for the spectral variation of a general matrix. One of our main results is
\begin{equation}\label{new}
\mathbb{D}_{2}\leq
\begin{cases}
\sqrt{n\Big(n-p+2\sqrt{n-p}\,\delta(E_{Q})+\frac{\delta(E_{Q})^{2}}{\|E_{Q}\|_{F}^{2}}\Big)\|E_{Q}\|_{F}^{\frac{2}{m}}+\frac{1}{n}|\tr(E)|^{2}}, &\text{if $\|E_{Q}\|_{F}<1$},\\
\sqrt{n\big(\sqrt{n-p}+\delta(E_{Q})\big)^{2}+\frac{1}{n}|\tr(E)|^{2}}, &\text{if $\|E_{Q}\|_{F}\geq 1$}.
\end{cases}
\end{equation}
In view of~\eqref{delta}, $\delta(E_{Q})$ involved in~\eqref{new} is $\delta(E_{Q})=\big(\|E_{Q}\|_{F}^{2}-\frac{1}{n}|\tr(E)|^{2}\big)^{\frac{1}{2}}$. Theoretical analysis shows that the new estimate~\eqref{new} is sharper than~\eqref{Song} (see Remark~\ref{comp} for details). Moreover, it is easy to check that~\eqref{new} will reduce to~\eqref{XXF} if $A$ is a normal matrix. That is, the new estimate~\eqref{new} also generalizes the existing one~\eqref{XXF}.

The rest of this paper is organized as follows. In Section~\ref{sec:pre}, we introduce several auxiliary estimates, which play an important role in our analysis. In Section~\ref{sec:main}, we present new upper bounds for the spectral variation of a general matrix.

\section{Preliminaries} \label{sec:pre}

For any square matrix $M$, the first lemma provides an upper bound for $\|\mathcal{L}(M)\|_{F}^{2}+\|\mathcal{U}(M)\|_{F}^{2}$~\cite[Lemma~3.1]{Xu2017}.

\begin{lemma}
Let $M$ be a square matrix. Then
\begin{equation}\label{LU}
\|\mathcal{L}(M)\|_{F}^{2}+\|\mathcal{U}(M)\|_{F}^{2}\leq\delta(M)^{2},
\end{equation}
where $\delta(\cdot)$ is defined by~\eqref{delta}.
\end{lemma}

The following lemma gives an upper bound for the spectral variation of a normal matrix~\cite[Theorem~3.6]{Xu2017}, which plays a key role in the subsequent analysis.

\begin{lemma}\label{lem:normal}
Let $A\in\mathbb{C}^{n\times n}$ be a normal matrix with spectrum $\{\lambda_{i}\}_{i=1}^{n}$, and let $\widetilde{A}=A+E\in\mathbb{C}^{n\times n}$ be a perturbed matrix with spectrum $\{\widetilde{\lambda}_{i}\}_{i=1}^{n}$. Then there exists a permutation $\pi$ of $\{1,\dots,n\}$ such that
\begin{equation}\label{Xu1}
\mathbb{D}_{2}\leq\sqrt{\|E\|_{F}^{2}+(n-1)\delta(E)^{2}},
\end{equation}
where $\delta(\cdot)$ is defined by~\eqref{delta}.
\end{lemma}

For any $A\in\mathbb{C}^{n\times n}$, it can be factorized as
\begin{equation}\label{Jordan}
A=Q\diag\big(J_{1},\ldots,J_{p}\big)Q^{-1},
\end{equation}
where $Q\in\mathbb{C}^{n\times n}$ is nonsingular, and each $J_{i}\in\mathbb{C}^{m_{i}\times m_{i}}$ $(\sum_{i=1}^{p}m_{i}=n)$ is a Jordan block with the form
\begin{displaymath}
J_{i}=\begin{pmatrix}
\lambda_{i} & 1 & 0 & \cdots & 0 \\
0 & \lambda_{i} & 1 & \cdots & 0 \\
\vdots & \vdots & \ddots & \ddots & \vdots \\
0 & 0 & \cdots &\lambda_{i} & 1 \\
0 & 0 & \cdots & 0 & \lambda_{i} \\
\end{pmatrix}.
\end{displaymath}
Let $0<\varepsilon\leq1$ be a parameter, and let
\begin{displaymath}
T=\diag\big(T_{1},\ldots,T_{p}\big),
\end{displaymath}
where $T_{i}=\diag\big(1,\varepsilon,\ldots,\varepsilon^{m_{i}-1}\big)$ for all $i=1,\ldots,p$. Then
\begin{equation}\label{decom}
T^{-1}Q^{-1}AQT=\diag\big(T_{1}^{-1}J_{1}T_{1},\ldots,T_{p}^{-1}J_{p}T_{p}\big)=\Lambda+\Omega,
\end{equation}
where $\Lambda=\diag\big(\lambda_{1}I_{m_{1}},\ldots,\lambda_{p}I_{m_{p}}\big)$, and $\Omega=\diag\big(\Omega_{1},\dots,\Omega_{p}\big)$ with
\begin{displaymath}
\Omega_{i}=\begin{pmatrix}
0 & \varepsilon & 0 & \cdots & 0 \\
0 & 0 & \varepsilon & \cdots & 0 \\
\vdots & \vdots & \ddots & \ddots & \vdots \\
0 & 0 & \cdots & 0 & \varepsilon \\
0 & 0 & \cdots & 0 & 0 \\
\end{pmatrix}\in\mathbb{C}^{m_{i}\times m_{i}}.
\end{displaymath}

We are now in a position to present the fundamental estimate of this paper.

\begin{lemma}
Let $A\in\mathbb{C}^{n\times n}$ be factorized as in~\eqref{Jordan}, and let $\widetilde{A}=A+E\in\mathbb{C}^{n\times n}$ be a perturbed matrix. Let
\begin{displaymath}
\Lambda=\diag\big(\lambda_{1}I_{m_{1}},\ldots,\lambda_{p}I_{m_{p}}\big) \quad \text{and} \quad T=\diag\big(T_{1},\ldots,T_{p}\big),
\end{displaymath}
where $T_{i}=\diag\big(1,\varepsilon,\ldots,\varepsilon^{m_{i}-1}\big)$ with $0<\varepsilon\leq 1$. Then, it holds that
\begin{equation}\label{main-est}
\|T^{-1}Q^{-1}\widetilde{A}QT-\Lambda\|_{F}^{2}\leq\mathscr{V}(\varepsilon),
\end{equation}
where
\begin{displaymath}
\mathscr{V}(\varepsilon):=\varepsilon^{2(1-m)}\delta(E_{Q})^{2}+2\varepsilon^{2}\sqrt{n-p}\,\delta(E_{Q})+(n-p)\varepsilon^{2}+\frac{1}{n}|\tr(E)|^{2}
\end{displaymath}
with $m=\max\limits_{1\leq i\leq p}m_{i}$ and $E_{Q}=Q^{-1}EQ$.
\end{lemma}

\begin{proof}
From~\eqref{decom}, we have
\begin{displaymath}
T^{-1}Q^{-1}\widetilde{A}QT-\Lambda=T^{-1}E_{Q}T+\Omega,
\end{displaymath}
which yields
\begin{equation}\label{rela1}
\|T^{-1}Q^{-1}\widetilde{A}QT-\Lambda\|_{F}^{2}=\|T^{-1}E_{Q}T\|_{F}^{2}+2\Retr(\Omega^{\ast}T^{-1}E_{Q}T)+\|\Omega\|_{F}^{2}.
\end{equation}
In what follows, we establish the upper bounds for $\|T^{-1}E_{Q}T\|_{F}^{2}$, $\Retr(\Omega^{\ast}T^{-1}E_{Q}T)$, and $\|\Omega\|_{F}^{2}$.

(i) Partitioning $E_{Q}$ into the block form $E_{Q}=(\widehat{E}_{ij})_{p\times p}$ with $\widehat{E}_{ij}\in\mathbb{C}^{m_{i}\times m_{j}}$, we have
\begin{displaymath}
\|T^{-1}E_{Q}T\|_{F}^{2}=\sum_{i=1}^{p}\sum_{j=1}^{p}\|T_{i}^{-1}\widehat{E}_{ij}T_{j}\|_{F}^{2}.
\end{displaymath}
Hence,
\begin{align*}
\|T^{-1}E_{Q}T\|_{F}^{2}&=\sum_{i=1}^{p}\sum_{j=1}^{p}\sum_{k=1}^{m_{i}}\sum_{\ell=1}^{m_{j}}\varepsilon^{2(\ell-k)}|(\widehat{E}_{ij})_{k,\ell}|^{2}\\
&\leq\varepsilon^{2(1-m)}\sum_{i\neq j}\|\widehat{E}_{ij}\|_{F}^{2}+\sum_{i=1}^{p}\big(\|\mathcal{D}(\widehat{E}_{ii})\|_{F}^{2}+\varepsilon^{2}\|\mathcal{U}(\widehat{E}_{ii})\|_{F}^{2}+\varepsilon^{2(1-m_{i})}\|\mathcal{L}(\widehat{E}_{ii})\|_{F}^{2}\big)\\
&\leq\varepsilon^{2(1-m)}\Bigg(\sum_{i\neq j}\|\widehat{E}_{ij}\|_{F}^{2}+\sum_{i=1}^{p}\|\mathcal{U}(\widehat{E}_{ii})\|_{F}^{2}+\sum_{i=1}^{p}\|\mathcal{L}(\widehat{E}_{ii})\|_{F}^{2}\Bigg)+\|\mathcal{D}(E_{Q})\|_{F}^{2}\\
&=\varepsilon^{2(1-m)}\big(\|E_{Q}\|_{F}^{2}-\|\mathcal{D}(E_{Q})\|_{F}^{2}\big)+\|\mathcal{D}(E_{Q})\|_{F}^{2}\\
&=\varepsilon^{2(1-m)}\|E_{Q}\|_{F}^{2}-\big(\varepsilon^{2(1-m)}-1\big)\|\mathcal{D}(E_{Q})\|_{F}^{2}.
\end{align*}
Note that
\begin{displaymath}
\|\mathcal{D}(E_{Q})\|_{F}^{2}\geq\frac{1}{n}|\tr(E)|^{2}.
\end{displaymath}
Thus,
\begin{equation}\label{rela2}
\|T^{-1}E_{Q}T\|_{F}^{2}\leq\varepsilon^{2(1-m)}\delta(E_{Q})^{2}+\frac{1}{n}|\tr(E)|^{2}.
\end{equation}

(ii) It is easy to see that
\begin{displaymath}
\Retr(\Omega^{\ast}T^{-1}E_{Q}T)=\re\sum_{i=1}^{p}\tr(\Omega_{i}^{\ast}T_{i}^{-1}\widehat{E}_{ii}T_{i})=\re\sum_{i=1}^{p}\sum_{j=2}^{m_{i}}\varepsilon(T_{i}^{-1}\widehat{E}_{ii}T_{i})_{j-1,j}.
\end{displaymath}
Due to $(T_{i}^{-1}\widehat{E}_{ii}T_{i})_{j-1,j}=\varepsilon(\widehat{E}_{ii})_{j-1,j}$, it follows that
\begin{align*}
\Retr(\Omega^{\ast}T^{-1}E_{Q}T)&=\re\sum_{i=1}^{p}\sum_{j=2}^{m_{i}}\varepsilon^{2}(\widehat{E}_{ii})_{j-1,j}\\
&\leq\varepsilon^{2}\sum_{i=1}^{p}\sum_{j=2}^{m_{i}}|(\widehat{E}_{ii})_{j-1,j}|\\
&\leq\varepsilon^{2}\sqrt{n-p}\Bigg(\sum_{i=1}^{p}\sum_{j=2}^{m_{i}}|(\widehat{E}_{ii})_{j-1,j}|^{2}\Bigg)^{\frac{1}{2}}\\
&\leq\varepsilon^{2}\sqrt{n-p}\Bigg(\sum_{i=1}^{p}\|\mathcal{U}(\widehat{E}_{ii})\|_{F}^{2}\Bigg)^{\frac{1}{2}}\\
&\leq\varepsilon^{2}\sqrt{n-p}\,\|\mathcal{U}(E_{Q})\|_{F}.
\end{align*}
Since $\|\mathcal{U}(E_{Q})\|_{F}\leq \delta(E_{Q})$ (see~\eqref{LU}), we obtain
\begin{equation}\label{rela3}
\Retr(\Omega^{\ast}T^{-1}E_{Q}T)\leq\varepsilon^{2}\sqrt{n-p}\,\delta(E_{Q}).
\end{equation}

(iii) In addition, we have
\begin{equation}\label{rela4}
\|\Omega\|_{F}^{2}=(n-p)\varepsilon^{2}.
\end{equation}
Combining~\eqref{rela1}--\eqref{rela4}, we can arrive at the estimate~\eqref{main-est}.
\end{proof}

\section{Main results} \label{sec:main}

In light of~\eqref{Xu1} and~\eqref{main-est}, we can derive the following estimate.

\begin{theorem}\label{main1.1}
Let $A\in\mathbb{C}^{n\times n}$ have the factorization~\eqref{Jordan}, and let $\widetilde{A}=A+E$, where $E\in\mathbb{C}^{n\times n}$ is a perturbation. Let $\{\lambda_{i}\}_{i=1}^{n}$ and $\{\widetilde{\lambda}_{i}\}_{i=1}^{n}$ be the spectra of $A$ and $\widetilde{A}$, respectively. Then there exists a permutation $\pi$ of $\{1,\ldots,n\}$ such that
\begin{equation}\label{up1.1}
\mathbb{D}_{2}\leq
\begin{cases}
\sqrt{n\Big(n-p+2\sqrt{n-p}\,\delta(E_{Q})+\frac{\delta(E_{Q})^{2}}{\|E_{Q}\|_{F}^{2}}\Big)\|E_{Q}\|_{F}^{\frac{2}{m}}+\frac{1}{n}|\tr(E)|^{2}}, &\text{if $\|E_{Q}\|_{F}<1$},\\
\sqrt{n\big(\sqrt{n-p}+\delta(E_{Q})\big)^{2}+\frac{1}{n}|\tr(E)|^{2}}, &\text{if $\|E_{Q}\|_{F}\geq 1$},
\end{cases}
\end{equation}
where $m=\max\limits_{1\leq i\leq p}m_{i}$ and $E_{Q}=Q^{-1}EQ$.
\end{theorem}

\begin{proof}
Observe that $\Lambda$ is a normal matrix with spectrum $\{\lambda_{i}\}_{i=1}^{n}$, and the spectrum of $T^{-1}Q^{-1}\widetilde{A}QT$ is $\{\widetilde{\lambda}_{i}\}_{i=1}^{n}$. Applying Lemma~\ref{lem:normal} to $\Lambda$ and $T^{-1}Q^{-1}\widetilde{A}QT$ yields
\begin{displaymath}
\mathbb{D}_{2}\leq\sqrt{n\|T^{-1}Q^{-1}\widetilde{A}QT-\Lambda\|_{F}^{2}-\frac{n-1}{n}|\tr(E)|^{2}}\leq\sqrt{n\mathscr{V}(\varepsilon)-\frac{n-1}{n}|\tr(E)|^{2}},
\end{displaymath}
where we have used the estimate~\eqref{main-est}. Take
\begin{displaymath}
\varepsilon=\begin{cases}
\|E_{Q}\|_{F}^{\frac{1}{m}}, &\text{if $\|E_{Q}\|_{F}<1$},\\
1, &\text{if $\|E_{Q}\|_{F}\geq 1$}.
\end{cases}
\end{displaymath}
Direct calculations yield
\begin{align*}
&\mathscr{V}\Big(\|E_{Q}\|_{F}^{\frac{1}{m}}\Big)=\bigg(n-p+2\sqrt{n-p}\,\delta(E_{Q})+\frac{\delta(E_{Q})^{2}}{\|E_{Q}\|_{F}^{2}}\bigg)\|E_{Q}\|_{F}^{\frac{2}{m}}+\frac{1}{n}|\tr(E)|^{2},\\
&\mathscr{V}(1)=\big(\sqrt{n-p}+\delta(E_{Q})\big)^{2}+\frac{1}{n}|\tr(E)|^{2}.
\end{align*}
Thus, the estimate~\eqref{up1.1} is valid.
\end{proof}

\begin{remark}\label{comp}\rm
If $\|E_{Q}\|_{F}<1$, then~\eqref{up1.1} reads
\begin{displaymath}
\mathbb{D}_{2}\leq\sqrt{n\bigg(n-p+2\sqrt{n-p}\,\delta(E_{Q})+\frac{\delta(E_{Q})^{2}}{\|E_{Q}\|_{F}^{2}}\bigg)\|E_{Q}\|_{F}^{\frac{2}{m}}+\frac{1}{n}|\tr(E)|^{2}}.
\end{displaymath}
Due to
\begin{displaymath}
n\frac{\delta(E_{Q})^{2}}{\|E_{Q}\|_{F}^{2}}\|E_{Q}\|_{F}^{\frac{2}{m}}+\frac{1}{n}|\tr(E)|^{2}=\frac{n\|E_{Q}\|_{F}^{2}-|\tr(E)|^{2}}{\|E_{Q}\|_{F}^{2}}\|E_{Q}\|_{F}^{\frac{2}{m}}+\frac{1}{n}|\tr(E)|^{2}\leq n\|E_{Q}\|_{F}^{\frac{2}{m}},
\end{displaymath}
it follows that
\begin{displaymath}
\mathbb{D}_{2}\leq\sqrt{n\big(n-p+2\sqrt{n-p}\,\delta(E_{Q})+1\big)\|E_{Q}\|_{F}^{\frac{2}{m}}}\leq\sqrt{n}\big(\sqrt{n-p}+1\big)\|E_{Q}\|_{F}^{\frac{1}{m}}.
\end{displaymath}
On the other hand, if $\|E_{Q}\|_{F}\geq1$, then~\eqref{up1.1} reads
\begin{displaymath}
\mathbb{D}_{2}\leq\sqrt{n\big(\sqrt{n-p}+\delta(E_{Q})\big)^{2}+\frac{1}{n}|\tr(E)|^{2}}.
\end{displaymath}
Then
\begin{align*}
\mathbb{D}_{2}\leq\sqrt{n\big(n-p+2\sqrt{n-p}\,\delta(E_{Q})+\|E_{Q}\|_{F}^{2}\big)}\leq\sqrt{n}\big(\sqrt{n-p}+1\big)\|E_{Q}\|_{F}.
\end{align*}
Hence, the estimate~\eqref{up1.1} is sharper than~\eqref{Song}.
\end{remark}

The next two estimates are based on the different constraints for $E_{Q}$.

\begin{theorem}\label{main1.2}
Under the assumptions of Theorem~{\rm\ref{main1.1}}, it holds that
\begin{equation}\label{up1.2}
\mathbb{D}_{2}\leq
\begin{cases}
\sqrt{n\big(n-p+2\sqrt{n-p}\,\delta(E_{Q})+1\big)\delta(E_{Q})^{\frac{2}{m}}+\frac{1}{n}|\tr(E)|^{2}}, &\text{if $\delta(E_{Q})<1$},\\
\sqrt{n\big(\sqrt{n-p}+\delta(E_{Q})\big)^{2}+\frac{1}{n}|\tr(E)|^{2}}, &\text{if $\delta(E_{Q})\geq 1$}.
\end{cases}
\end{equation}
\end{theorem}

\begin{proof}
Take
\begin{displaymath}
\varepsilon=\begin{cases}
\delta(E_{Q})^{\frac{1}{m}}, &\text{if $\delta(E_{Q})<1$},\\
1, &\text{if $\delta(E_{Q})\geq 1$}.
\end{cases}
\end{displaymath}
Direct computation yields
\begin{displaymath}
\mathscr{V}\Big(\delta(E_{Q})^{\frac{1}{m}}\Big)=\big(n-p+2\sqrt{n-p}\,\delta(E_{Q})+1\big)\delta(E_{Q})^{\frac{2}{m}}+\frac{1}{n}|\tr(E)|^{2}.
\end{displaymath}
Similarly to Theorem~\ref{main1.1}, one can show that the estimate~\eqref{up1.2} holds.
\end{proof}

\begin{theorem}\label{main1.3}
Under the assumptions of Theorem~{\rm\ref{main1.1}}, it holds that
\begin{equation}\label{up1.3}
\mathbb{D}_{2}\leq
\begin{cases}
\sqrt{mn\Big(\frac{n-p+2\sqrt{n-p}\,\delta(E_{Q})}{m-1}\Big)^{1-\frac{1}{m}}\delta(E_{Q})^{\frac{2}{m}}+\frac{1}{n}|\tr(E)|^{2}}, &\text{if ${\rm C}_{1}$ holds},\\
\sqrt{n\big(\sqrt{n-p}+\delta(E_{Q})\big)^{2}+\frac{1}{n}|\tr(E)|^{2}}, &\text{if ${\rm C}_{2}$ holds},
\end{cases}
\end{equation}
where
\begin{align*}
&{\rm C}_{1}: n-p+2\sqrt{n-p}\,\delta(E_{Q})>(m-1)\delta(E_{Q})^{2},\\
&{\rm C}_{2}: n-p+2\sqrt{n-p}\,\delta(E_{Q})\leq(m-1)\delta(E_{Q})^{2}.
\end{align*}
\end{theorem}

\begin{proof}
We first note that $A$ is diagonalizable if and only if $n=p$ (or $m=1$).

(i) If $A$ is diagonalizable, then $T=I_{n}$, $n=p$, and $m=1$. In this case, the estimate~\eqref{main-est} reduces to
\begin{displaymath}
\|Q^{-1}\widetilde{A}Q-\Lambda\|_{F}^{2}\leq\|E_{Q}\|_{F}^{2}.
\end{displaymath}
An application of Lemma~\ref{lem:normal} yields
\begin{equation}\label{diag}
\mathbb{D}_{2}\leq\sqrt{n\|Q^{-1}\widetilde{A}Q-\Lambda\|_{F}^{2}-\frac{n-1}{n}|\tr(E)|^{2}}\leq\sqrt{n\|E_{Q}\|_{F}^{2}-\frac{n-1}{n}|\tr(E)|^{2}}.
\end{equation}

(ii) If $A$ cannot be diagonalized, then $n>p$ and $m>1$. Direct calculation yields
\begin{displaymath}
\mathscr{V}'(\varepsilon)=2\varepsilon\bigg(n-p+2\sqrt{n-p}\,\delta(E_{Q})-\frac{(m-1)\delta(E_{Q})^{2}}{\varepsilon^{2m}}\bigg).
\end{displaymath}
Here, $\mathscr{V}'(\varepsilon)$ denotes the derivative of $\mathscr{V}(\varepsilon)$ with respect to $\varepsilon$. It is easy to check that
\begin{displaymath}
\begin{cases}
\mathscr{V}'(\varepsilon)>0, &\text{if $\varepsilon>\Big(\frac{(m-1)\delta(E_{Q})^{2}}{n-p+2\sqrt{n-p}\,\delta(E_{Q})}\Big)^{\frac{1}{2m}}$},\\
\mathscr{V}'(\varepsilon)<0, &\text{if $0<\varepsilon<\Big(\frac{(m-1)\delta(E_{Q})^{2}}{n-p+2\sqrt{n-p}\,\delta(E_{Q})}\Big)^{\frac{1}{2m}}$}.
\end{cases}
\end{displaymath}
Take
\begin{displaymath}
\varepsilon=\begin{cases}
\Big(\frac{(m-1)\delta(E_{Q})^{2}}{n-p+2\sqrt{n-p}\,\delta(E_{Q})}\Big)^{\frac{1}{2m}}, &\text{if $n-p+2\sqrt{n-p}\,\delta(E_{Q})>(m-1)\delta(E_{Q})^{2}$},\\
1, &\text{if $n-p+2\sqrt{n-p}\,\delta(E_{Q})\leq(m-1)\delta(E_{Q})^{2}$}.
\end{cases}
\end{displaymath}
Direct computation yields
\begin{align*}
\mathscr{V}\Bigg(\bigg(\frac{(m-1)\delta(E_{Q})^{2}}{n-p+2\sqrt{n-p}\,\delta(E_{Q})}\bigg)^{\frac{1}{2m}}\Bigg)&=m\bigg(\frac{n-p+2\sqrt{n-p}\,\delta(E_{Q})}{m-1}\bigg)^{1-\frac{1}{m}}\delta(E_{Q})^{\frac{2}{m}}\\
&\quad+\frac{1}{n}|\tr(E)|^{2}.
\end{align*}
The rest of this proof is similar to Theorem~\ref{main1.1}.
\end{proof}

\begin{remark}\rm
If $A$ is diagonalizable, the condition ${\rm C}_{2}$ will be satisfied. From~\eqref{up1.3}, we have
\begin{displaymath}
\mathbb{D}_{2}\leq\sqrt{n\,\delta(E_{Q})^{2}+\frac{1}{n}|\tr(E)|^{2}},
\end{displaymath}
which coincides with~\eqref{diag}. That is,~\eqref{up1.3} has contained the diagonalizable case.
\end{remark}

\begin{remark}\rm
In particular, if $A$ is normal, then $Q$ can be chosen as a unitary matrix. In this case, the estimates~\eqref{up1.1}--\eqref{up1.3} all reduce to
\begin{displaymath}
\mathbb{D}_{2}\leq\sqrt{n\,\delta(E)^{2}+\frac{1}{n}|\tr(E)|^{2}},
\end{displaymath}
which is exactly~\eqref{Xu1}.
\end{remark}

\begin{example}
Let
\begin{displaymath}
A=\begin{pmatrix}
a+b\mathbf{i} & 1 & 0 \\
0 & a+b\mathbf{i} & 0 \\
0 & 0 & a+b\mathbf{i}
\end{pmatrix} \quad \text{and} \quad E=\begin{pmatrix}
0.0098 & 0 & 0 \\
0 & 0.01 & 0 \\
0 & 0 & 0.0102
\end{pmatrix},
\end{displaymath}
where $a\in\mathbb{R}$, $b\in\mathbb{R}$, and $\mathbf{i}=\sqrt{-1}$. In this case,
\begin{displaymath}
\mathbb{D}_{2}\equiv\sqrt{(0.0098)^{2}+(0.01)^{2}+(0.0102)^{2}}\approx 0.017322817323.
\end{displaymath}
The upper bounds in~\eqref{Song}, \eqref{up1.1}, \eqref{up1.2}, and~\eqref{up1.3} are listed below.
\begin{table}[!htbp]
\centering
\begin{tabular}{@{} cc @{}}
\toprule
\ \ \texttt{Estimate}  &  \qquad \texttt{Upper bound for $\mathbb{D}_{2}$} \ \ \\
\midrule
\ \ \eqref{Song}  &  \qquad 0.455931801780 \ \ \\
\ \ \eqref{up1.1}  &  \qquad 0.228717520806 \ \ \\
\ \ \eqref{up1.2}  &  \qquad 0.044693805777 \ \ \\
\ \ \eqref{up1.3}  &  \qquad 0.044693805018 \ \ \\
\bottomrule
\end{tabular}
\medskip
\caption{\small The upper bounds in~\eqref{Song} and~\eqref{up1.1}--\eqref{up1.3}.}
\label{tab:complex}
\end{table}

\vskip -0.5cm

Table~\ref{tab:complex} displays that the new upper bounds in~\eqref{up1.1}--\eqref{up1.3} are smaller than that in~\eqref{Song}.
\end{example}

Under the assumptions of Lemma~\ref{lem:normal}, if the original matrix is Hermitian, then the following estimate (see~\cite[Theorem~4.2]{Xu2017}) holds:
\begin{equation}\label{Xu2}
\mathbb{D}_{2}\leq\sqrt{\|E\|_{F}^{2}+\delta(E)^{2}}.
\end{equation}
In what follows, we consider a special case that the eigenvalues of $A$ are all real. In such a case, we can derive more accurate estimates for $\mathbb{D}_{2}$ based on~\eqref{Xu2}, which are presented in the following three theorems.

\begin{theorem}\label{main2.1}
Let $A\in\mathbb{C}^{n\times n}$ be factorized as in~\eqref{Jordan}, and let $\widetilde{A}=A+E\in\mathbb{C}^{n\times n}$ be a perturbed matrix with spectrum $\{\widetilde{\lambda}_{i}\}_{i=1}^{n}$. If the eigenvalues $\{\lambda_{i}\}_{i=1}^{n}$ of $A$ are all real, then there exists a permutation $\pi$ of $\{1,\ldots,n\}$ such that
\begin{equation}\label{up2.1}
\mathbb{D}_{2}\leq
\begin{cases}
\sqrt{2\Big(n-p+2\sqrt{n-p}\,\delta(E_{Q})+\frac{\delta(E_{Q})^{2}}{\|E_{Q}\|_{F}^{2}}\Big)\|E_{Q}\|_{F}^{\frac{2}{m}}+\frac{1}{n}|\tr(E)|^{2}}, &\text{if $\|E_{Q}\|_{F}<1$},\\
\sqrt{2\big(\sqrt{n-p}+\delta(E_{Q})\big)^{2}+\frac{1}{n}|\tr(E)|^{2}}, &\text{if $\|E_{Q}\|_{F}\geq 1$}.
\end{cases}
\end{equation}
\end{theorem}

\begin{theorem}\label{main2.2}
Under the assumptions of Theorem~{\rm\ref{main2.1}}, it holds that
\begin{equation}\label{up2.2}
\mathbb{D}_{2}\leq
\begin{cases}
\sqrt{2\big(n-p+2\sqrt{n-p}\,\delta(E_{Q})+1\big)\delta(E_{Q})^{\frac{2}{m}}+\frac{1}{n}|\tr(E)|^{2}}, &\text{if $\delta(E_{Q})<1$},\\
\sqrt{2\big(\sqrt{n-p}+\delta(E_{Q})\big)^{2}+\frac{1}{n}|\tr(E)|^{2}}, &\text{if $\delta(E_{Q})\geq 1$}.
\end{cases}
\end{equation}
\end{theorem}

\begin{theorem}\label{main2.3}
Under the assumptions of Theorem~{\rm\ref{main2.1}}, it holds that
\begin{equation}\label{up2.3}
\mathbb{D}_{2}\leq
\begin{cases}
\sqrt{2m\Big(\frac{n-p+2\sqrt{n-p}\,\delta(E_{Q})}{m-1}\Big)^{1-\frac{1}{m}}\delta(E_{Q})^{\frac{2}{m}}+\frac{1}{n}|\tr(E)|^{2}}, &\text{if ${\rm C}_{1}$ holds},\\
\sqrt{2\big(\sqrt{n-p}+\delta(E_{Q})\big)^{2}+\frac{1}{n}|\tr(E)|^{2}}, &\text{if ${\rm C}_{2}$ holds},
\end{cases}
\end{equation}
where ${\rm C}_{1}$ and ${\rm C}_{2}$ are given in Theorem~{\rm\ref{main1.3}}.
\end{theorem}

\begin{example}
Let
\begin{displaymath}
A=\begin{pmatrix}
\lambda & 1 & 0 \\
0 & \lambda & 1 \\
0 & 0 & \lambda
\end{pmatrix} \quad \text{and} \quad E=\begin{pmatrix}
0.099 & 0 & 0 \\
0 & 0.101 & -0.0001 \\
0 & 0 & 0.1
\end{pmatrix},
\end{displaymath}
where $\lambda\in\mathbb{R}$. In this example, it holds that
\begin{displaymath}
\mathbb{D}_{2}\equiv\sqrt{(0.099)^{2}+(0.101)^{2}+(0.1)^{2}}\approx 0.173210854163.
\end{displaymath}
The upper bounds in~\eqref{Song}, \eqref{up2.1}, \eqref{up2.2}, and~\eqref{up2.3} are listed below.
\begin{table}[!htbp]
\centering
\begin{tabular}{@{} cc @{}}
\toprule
\ \ \texttt{Estimate}  &  \qquad \texttt{Upper bound for $\mathbb{D}_{2}$} \ \ \\
\midrule
\ \ \eqref{Song}  &  \qquad 2.330923594272 \ \ \\
\ \ \eqref{up2.1}  &  \qquad 1.129360191939 \ \ \\
\ \ \eqref{up2.2}  &  \qquad 0.325303334100 \ \ \\
\ \ \eqref{up2.3}  &  \qquad 0.325303282160 \ \ \\
\bottomrule
\end{tabular}
\medskip
\caption{\small The upper bounds in~\eqref{Song} and~\eqref{up2.1}--\eqref{up2.3}.}
\label{tab:real}
\end{table}

\vskip -0.5cm

From Table~\ref{tab:real}, one can see that the new estimates~\eqref{up2.1}--\eqref{up2.3} are sharper than~\eqref{Song}.
\end{example}

\begin{remark}
Define
\begin{displaymath}
\kappa_{2}(Q):=\|Q^{-1}\|_{2}\|Q\|_{2} \quad \text{and} \quad \mathbb{D}_{\infty}:=\max_{1\leq i\leq n}\big|\widetilde{\lambda}_{\pi(i)}-\lambda_{i}\big|.
\end{displaymath}
Using
\begin{displaymath}
\|E_{Q}\|_{F}\leq\min\Big\{\sqrt{{\rm rank}(E)}\|E_{Q}\|_{2},\,\kappa_{2}(Q)\|E\|_{F}\Big\},
\end{displaymath}
one can derive some deductive estimates for $\mathbb{D}_{2}$. Furthermore, using the relation $\mathbb{D}_{\infty}\leq\mathbb{D}_{2}$, one can readily obtain the corresponding estimates for $\mathbb{D}_{\infty}$.
\end{remark}

\bibliographystyle{amsplain}

\end{document}